\documentclass[12pt]{amsart}
\usepackage{amssymb}
\numberwithin{equation}{section}
\usepackage{mathrsfs}
\usepackage{amsfonts}
\usepackage{bm,upgreek,verbatim}
\usepackage{graphics, latexsym, tabularx, shapepar, enumerate}
\usepackage[all]{xy} 
\usepackage[small,nohug]{diagrams}
\diagramstyle[labelstyle=\scriptstyle]

\input xy
\xyoption{all}

\theoremstyle{plain}
\newtheorem{prop}{Proposition}
\newtheorem{Theo}[prop]{Theorem}

\newtheorem{princ}[prop]{Principle}
\newtheorem{que}[prop]{Question}
\newtheorem{lemm}[prop]{Lemma}

\theoremstyle{definition}
\newtheorem{defi}[prop]{Definition}

\theoremstyle{remark}

\renewcommand{\P}{{\mathbb P}}
\newcommand{\F}{{\mathbb F_{p}}}

\newcommand{\Q}{{\mathbb Q_{p}}}

\newcommand{\Z}{{\mathbb Z}}
\newcommand{\N}{{\mathbb N}}
\newcommand{\C}{{\mathbb C}}

\newcommand{\bQ}{\mathbb Q}

\newcommand{\Spec}{{\rm Spec}}

\begin{document}

\title[Rationally connected varieties over $\bQ_p^{nr}$]{Rationally connected varieties over the maximally unramified extension of $p$-adic fields}

\author{Bradley Duesler}
\address{Department of Mathematics \\
Rice University\\
Houston, TX 77005}
\curraddr{ Voya Financial \\ Malvern, PA 19355}
\email{bradley.duesler@voya.com}
\thanks {B.D. was supported by NSF grants DMS-0134259 and  DMS-0240058 and A.K. was supported by NSF grants  DMS-0502170}

\author{Amanda Knecht}
\address{Department of Mathematics and Statistics\\
Villanova University \\
Villanova, PA 19085}
\email{Amanda.Knecht@villanova.edu}
\thanks{The second author was supported in part by NSF Grant DMS-0502170}

\begin{abstract}
A result of Graber, Harris, and Starr shows that a rationally
connected variety defined over the function field of a curve over
the complex numbers always has a rational point. Similarly, a separably rationally
connected variety over a finite field or the function field of a
curve over any algebraically closed field will have a rational
point. Here we show that rationally connected varieties over the
maximally unramified extension of the $p$-adics usually, in a precise
sense, have rational points. This result is in the spirit of Ax and
Kochen's result saying that the $p$-adics are usually $C_{2}$ fields.
The method of proof utilizes a construction from mathematical
logic called the ultraproduct.

\end{abstract}

\maketitle

\section{Introduction}
Let $X$ be a proper, smooth variety over a field $K$ and $\overline{K}$ an algebraic closure of $K$.  A guiding principle in the study of $K$-rational points on $X$ is given by Koll\'{a}r  (\cite{MR1440180} IV.6.3]):

\begin{princ}
 If $\overline{X} = X \times_{\Spec( K)}  \Spec (\overline{K})$ is rationally connected, then $X$ should have lots of $K$-points, at least if $K$ is nice (e.g. $K$ is a finite field, a function field of a curve, or a sufficiently large number field).
\end{princ}

The term ``nice" has since been replaced by many  with the term  quasi-algebraically closed.  
 A field $K$ is said to be quasi-algebraically closed or $C_1$ if every
homogeneous polynomial over $K$ with degree less than the number of
variables has a nontrivial solution in $K$. Some well known examples
are finite fields, function fields in one variable over an
algebraically closed field,  the field of Laurent series over an algebraically closed field and the maximal unramified
extension of $p$-adic fields, $\bQ_p^{nr}.$

A homogeneous polynomial in $n$ variables over $K$ defines a hypersurface in projective $(n-1)$-space $\P^{n-1}$. The hypersurface
associated to a form with degree less than the number of variables,
when smooth, is a Fano variety, thus is rationally connected (see \cite{Ca}, \cite{KMM}).
  Since smooth rationally connected hypersurfaces defined over quasi-algebraically closed fields always have a $K$-rational point, it is natural to ask:
  \begin{que} (\cite{Witt} 1.11)
Let  $X$ be a proper, smooth separably rationally connected variety over a field $K$ where $K$ is a quasi-algebraically closed field.  Is $X(K)= \emptyset$?
 \end{que}

  \noindent Affirmative answers to this question have been given when:
  \begin{enumerate}
  \item $K$ is the function field of a curve defined over an algebraically closed field of characteristic zero  \cite{MR1937199}.  
  \item  $K$ is the function field of  a curve  defined over an algebraically closed field of positive characteristic \cite{MR1981034}.
  \item $K$ is a finite field \cite{MR1943746}.
    \item $K=k((t))$ is the field of Laurent series over an algebraically closed field \cite{MR2757627}.
  \item If $X$ is a smooth, proper, rational \textit{surface} over a quasi-algebraically closed field $K$, then $X(K)\neq \emptyset$ \cite{Manin} , \cite{CT1}.
  \end{enumerate}
  
Colliot-Th{\'e}l{\`e}ne and Madore have shown that there exist  fields $K$ of cohomological dimension $1$ and del Pezzo surfaces $X_d$ of degrees $d=2,3,4$ such that $X_d(K)= \emptyset$ \cite{CTM}.  In particular, these fields are examples of fields of cohomological dimension $1$ which are not $C_1$.  This  seems to rule out a cohomological proof that separably rationally connected varieties over $C_1$ fields have points.

Let us recall
some basic facts about rational connectivity.  
Suppose that $X$ is a smooth, projective variety defined over an arbitrary field $K$.  We say that $X$ is \textit{separably rationally connected} if there is a variety $Y$ and a morphism $u:Y \times \P^1 \rightarrow X$ such that   $$u^{(2)}: (Y \times \P^1) \times (Y \times \P^1) \rightarrow X \times X$$ is dominant and smooth at the generic point.  
If the field $K$ is algebraically closed, we can simplify the definition and say that $X$ is \textit{separably rationally connected} if there is a rational curve, called a very free curve, $f: \P^{1} \rightarrow X$ such
that $f^{*}T_{X}$ is ample  (\cite{MR1440180} IV.3.7).  
Over an arbitrary field $K$ we say that $X$ is \textit{rationally connected} if there is a family of proper algebraic curves $g: U \rightarrow Y$whose geometric fibers are irreducible rational curves with cycle morphism $u:U \rightarrow X$  such that $u^{(2)}$ is dominant.
If the field $K$ is uncountable and  algebraically closed, then we say that $X$ is \textit{rationally connected} if for very general closed points $x_1,x_2 \in X$ there is a morphism  $f: \P^{1} \rightarrow X$
such that $x_1, x_2\in f(\P^1)$ (\cite{MR1440180} IV.3.6).    Over any field of characteristic zero, the notions of rationally connected and separably rationally connected are equivalent   (\cite{MR1440180} IV.3.3).

The proof that rational connectivity and separable rational connectivity  are equivalent over characteristic zero relies in some way on
generic smoothness, which fails in characteristic $p$. In fact, in
positive characteristic there are smooth, projective varieties for which one can find a rational curve through  any two closed points, but the variety does not contain any very free curves 
\cite{MR0526513}. One can think of this as rationally connected varieties containing lots of rigid rational curves while separably rationally connected varieties have rational curves that freely deform.   Another example of a rationally connected but not separably rationally connected  variety  over a field of positive characteristic is give by Koll\'{a}r (\cite{MR1440180} V.5.19).

Because of the subtleties between rationally and separably rationally connected over characteristic $p$, we had to be careful when stating Question 2.  But in this paper we are considering varieties defined over fields of characteristic zero, so the terms are interchangeable.   The question we consider in this article is
whether or not a smooth, projective, rationally connected variety
over the maximal unramified extension of the $p$-adics,
$\mathbb{Q}_{p}^{nr}$, has a rational point. Lang's theorem asserts
that this is true for Fano hypersurfaces \cite{MR0046388}. Here we prove a partial
result.

\begin{Theo}\label{thm_1}
Fix a numerical polynomial $P$. There is a finite
set of exceptional primes $e(P)$, depending only on $P$, so that if $X$   is a smooth, projective,
rationally connected variety defined over
$\mathbb{Q}_{p}^{nr}$ with Hilbert polynomial $P$, then $X(\mathbb{Q}_{p}^{nr}) \neq \emptyset$ as long as $p \notin e(P)$.
\end{Theo}

A polynomial $P(z) \in \mathbb{Q}(z)$ is called a \textit{numerical polynomial} if $P(n)$  is an integer for all sufficiently large integers $n$.

This theorem is similar to Ax and Kochen's theorem \cite{MR0184930} that
the $p$-adic number fields are almost $C_{2}$.   Emil Artin conjectured that the $p$-adic fields $\Q$ are
$C_{2}$. In general, a $C_{i}$ field K is one for which any form in
$K[x_{1}, \ldots , x_{n}]_{d}$ with $n > d^{i}$ has a nontrivial
zero. In \cite{MR0197450} Terjanian found a counterexample to Artin's
conjecture, see for instance \cite{MR0344216}. However, using the methods
of mathematical logic, Ax and Kochen were able to show that
$\Q$ is almost $C_{2}$ in the following sense.  

\begin{Theo}\label{thm_AK}(Ax and Kochen)
  Fix an integer $d > 0$. Then there exists a finite number of primes $p_{0}, \ldots ,p_{m}$ such that for all forms
  $f \in \Q[x_{1}, \ldots , x_{n}]_{d}$ with $n > d^{2}$ and $p \neq p_{0}, \ldots ,p_{m}$, $f$ represents zero over $\Q$.
\end{Theo}
Their method of proof uses mathematical logic to make precise the
analogy that $\Q$ is like ${\F}((t))$. Then using the fact
that the field ${\F}((t))$ is $C_{2}$ \cite{MR0191897} is enough for Ax and Kochen to
conclude the above theorem. We similarly make an analogy between the asymptotic properties of $\bQ_p^{nr}$ when $p$ goes to infinity and the properties of  $\C((t))$, the field of Laurent expansions over the complex numbers.  Then we use the fact that every rationally connected variety over $\C((t))$ contains a $\C((t))$-point \cite{MR2757627}.

It should be noted that in a recent paper, Denef proves Theorem \ref{thm_AK} using only Algebraic Geometry \cite{denef}. \\

\noindent {\bf Acknowledgments:}  We would like to thank Brendan Hassett for many 
helpful conversations concerning the topics in this paper. We are also grateful to Jean-Louis Colliot-Th{\'e}l{\`e}ne, Keith Conrad, Olivier Wittenberg, and the referee for their insightful comments.

\section{Model Theory and Algebraic Geometry}
The main tool from Model Theory we will use is the ultraproduct.  A more
thorough introduction to ultrafilters and ultraproducts is given in
\cite{MR0568318}.

\begin{defi}
Let S be a set and let $\Sigma$ be a collection of non-empty subsets
of S. Then $\Sigma$ is called a \textit{non-principal filter} if the following hold:
\begin{enumerate}
\item[(1)] $S_{1}, S_{2} \in \Sigma$ implies $S_{1} \cap S_{2} \in \Sigma$
\item[(2)] $S_{1} \in \Sigma$ and $S_{2} \supset S_{1}$ implies $S_{2} \in \Sigma$ 
\item[(3)] For each $s \in S$ there is a set $S_{1} \in \Sigma$ such that $s \notin S_{1}.$

\end{enumerate}
\noindent $\Sigma$ is called a  \textit{non-principal ultrafilter}  if it is maximal among the class of all non-principal filters 
on $S$, or equivalently:
\begin{enumerate}
\item[(4)] $S_{1} \notin \Sigma$ implies $S - S_{1} \in \Sigma$.
\end{enumerate}
\end{defi}
Conditions $1,2,$ and $4$ define an \textit{ultrafilter}.

A simple, but important property of ultrafilters to keep in mind is
that if $S$ is the disjoint union of subsets $S_{1},\ldots,S_{n}$,
then precisely one of these subsets is in $\Sigma$.  This
observation follows from properties 1 and 4. Namely, at least one of
the $S_{i}$ is in $\Sigma$ by property 4. Moreover, two disjoint
subsets cannot both be in $\Sigma$ since then so would their
intersection, but $\Sigma$ consists only of nonempty subsets of $S$.

Given any subset $S_{0} \subset S$ it will be useful to know if we
can find a  non-principal ultrafilter on S containing $S_{0}$. Certainly, if
$S_{0}$ is a finite set, then properties 3 and 4 of the definition above  will prevent
us from finding a non-principal ultrafilter containing $S_{0}$. However, this is the
only obstruction as the  lemma below asserts.

\begin{lemm}\label{lem_ultra}
Given any infinite subset $S_{0} \subset S$, there exists a non-principal
ultrafilter containing $S_{0}$.
\end{lemm}

\begin{proof}
Let $\Sigma$ consist of all the subsets of $S$ that contain all but a finite number of points in $S_0$.  It is easy to check that  $\Sigma$ is a non-principal filter on $S$ containing $S_0$.  The non-principal ultrafilter desired is any maximal filter containing $\Sigma$.
\end{proof}

Our usage of ultrafilters will be for an auxiliary construction
called the ultraproduct. In particular, given a collection of fields
indexed by a set $S$, and an ultrafilter $\Sigma$ on $S$, we will
construct a new field via the ultraproduct. We will use a similar
construction for modules.
\begin{defi}
Given an ultrafilter $\Sigma$ on $S$ and a collection of rings
$\{R_{i}\}_{i \in S}$ we can form a new ring denoted
\[
\prod_{i \in S}R_{i}/\Sigma
\]
defined by componentwise addition and multiplication under the
equivalence condition that $a,b \in \prod_{i \in S}R_{i}$ are
equivalent if they agree on a set of indices in $\Sigma$. This new
ring is called the \textit{ultraproduct} of the $R_{i}$'s with respect to
$\Sigma$.  The same definition can be made for groups, modules, etc.  
\end{defi}
One of the nice aspects of ultraproducts is that the ultraproduct of fields is also a field. Moreover, statements in the language of fields can be transferred between the ultraproduct and its components which leads to the fundamental property of ultraproducts.  Let $\{k_p\}_{p \in S}$ be a collection of fields. Then {\L}o\v{s}'s Theorem (\cite{FieldArith}  7.7.1) applied to the particular cases of an  ultraproduct of fields can be stated as: 
\begin{Theo}({\L}o\v{s})\label{Los} A first-order formula in the language of rings is true in the ultraproduct of fields $\prod_{p \in S}k_p/\Sigma$ if and only if the set of indices $p$ such that the formula is true in the field $k_p$ is a member of $\Sigma$.
\end{Theo}
Intuitively, a first-order formula is a
formula that only quantifies  over elements of the field, not over subsets, sets of subsets, etc.
To get a feeling for why {\L}o\v{s}'s theorem is true, even for structures more general than fields, consider the following lemma.

\begin{lemm}\label{lem_freemodule}  Let $N$ be a positive integer.
For each $i \in S$, let $M_{i}$ be a free module of rank less than
$N$ over a ring $R_{i}$. Then an ultraproduct of the $M_{i}$'s is a
free module of rank less than $N$ over the corresponding
ultraproduct of the $R_{i}$'s.
\end{lemm}
\begin{proof}
First, assume that the rank of the $M_{i}$'s are all $m > 0$. Then
for any ultrafilter $\Sigma$ on $S$
\[
M := \prod_{i \in S} M_{i} /\Sigma
\]
will be a free module of rank $m$ over
\[
R := \prod_{i \in S} R_{i} /\Sigma.
\]
To see this, let $e_{i1},\ldots,e_{im}$ be an $R_{i}$ basis for
$M_{i}$. Then note that $M$ has basis $(e_{i1})_{i \in
S},\ldots,(e_{im})_{i \in S}$ over $R$.

Now generally, consider the subsets $S_{k} \in S$ consisting of
those $i \in S$ such that the rank of $M_{i}$ is $k$. Then S is the
disjoint union of $S_{1},\ldots,S_{N}$. By the remarks on the
definition of ultrafilter, there is only one such subset contained
in $\Sigma$, say $S_{m} \in \Sigma$. It follows that $M$ has rank
$m$ over $R$.
\end{proof}
What makes this lemma work is the boundedness of the statement (that the rank is less than $N$). The similar statement that the ultraproduct of finite rank free modules is finite rank is actually false (say if the rank of the free modules keeps increasing). {\L}o\v{s}'s theorem does not apply to such a statement because it is not a first-order statement.

Next, we  develop some basic algebraic geometry over a general ultraproduct of
fields $F = \prod_{i \in S}F_{i}/\Sigma$ of characteristic zero.  
 Suppose we are given a
scheme $X$ of finite type over $F$. There is a natural process to
obtain schemes $X_{i}$ of finite type over $F_{i}$, and for almost
every $i \in S$, $X_{i}$ is nicely related to $X$. However, the
$X_{i}$ are not unique.

Let's assume first that $X$ is an affine scheme corresponding to the
$F$-algebra
\[
F[x_{0},\ldots,x_{n}]/I(X).
\]
Suppose that $f_1,\ldots, f_k$ are generators for $I(X)$. We may
write each generator as
\[
f_{j} = \sum a_{j,I}x^{I}, \hspace{2mm}  I \in \N^{n+1}.
\]
Let $(a_{j,I}^{i})_{i\in S} \in \prod_{i \in S} F_{i}$ be a
representative for $a_{j,I}$. Setting
\[
f_{j}^{i} = \sum a_{j,I}^{i} x^{I},
\]
we define $X_{i}$ as the affine scheme associated to the ideal
generated by the $f_{j}^{i}$. These schemes are not unique as they
depend on the choice of representatives for the $a_{j,I}$.

We can perform a similar construction in reverse. Namely given schemes $X_i$ defined over $F_i$ for each $i\in S$ , we can defined their ultraproduct $X=\prod_{i \in S} X_{i} /\Sigma$ by taking the $f_j^i$ and lifting them to $f \in F[x_0, \ldots, x_n]$.  This new object is not necessarily pretty, for example if the degrees of the $f^i_j$ are not bounded.  However, we will see that under certain circumstances the ultraproduct is a scheme of finite type and other nice properties of the $X_i's$ will be inherited by $X$.

Let $X \subseteq \P^n$  be a projective variety defined over a field $F$ with homogeneous ideal $J(X)$, and let $S(X)= F[x_0, \ldots, x_n] / J(X)$  denote its homogeneous
coordinate ring. For each integer $\ell$, we define the Hilbert function $\varphi_X$ of $X$ by,
$$\varphi_X(\ell) = \dim_F S(X)_\ell.$$
It is a theorem of Hilbert and Sere (\cite{Hart} I.7.5) that there exists a unique numerical polynomial $P(z) \in \mathbb{Q}[z]$ such that $\varphi_X(\ell)= P(\ell)$ for all $l \gg 0$.  By definition, the degree of the Hilbert polynomial is the dimension of the variety $X$.
Chardin and Moreno-Soc{\'{\i}}as characterize, in terms of their coefficients, which numerical polynomials are
Hilbert polynomials of some projective scheme \cite{MR1948099}.
\begin{lemm}\label{lemma_hilb}
Given a collection of
projective varieties $X_{i}\subset \P_{F_{i}}^n$ all with Hilbert polynomial $P$,   the ultraproduct $X$ is a projective variety in $\mathbb{P}_{F}^{n}$ with Hilbert
polynomial $P$.
\end{lemm}

\begin{proof}
Let $J_{i} \subset F_{i}[x_{0},\ldots,x_{n}]$ be the homogeneous
ideal of $X_{i} \subset \mathbb{P}_{F_{i}}^{n}$. For each degree $d
> 0$ consider the $F_{i}$-vector space $J_{i,d}$ of the homogeneous
polynomials of degree $d$ in $J_{i}$. 

 Now define
\[
J_{d} := \prod_{i \in S} J_{i,d} /\Sigma .
\]
It is a property of the Hilbert polynomial that for sufficiently large $d$ the rank of $J_{i,d}$ is the same for each $i\in S$, i.e. the Hilbert functions of the $X_i$ are equal for sufficiently large $d$ (\cite{MR1440180} I.1.5).  Then, the proof of Lemma \ref{lem_freemodule} shows that for $d\gg0$ the rank of $J_d$ equals the rank of $J_{i,d}$.
This yields a homogeneous ideal
\[
J := \bigoplus_{d > 0} J_{d} \subset F[x_0, \ldots , x_n].
\]
The corresponding projective variety $X$ denoted by
\[
X := \prod_{i \in S} X_{i} /\Sigma
\]
has Hilbert Polynomial $P$.
\end{proof}

Other nice results on properties of varieties preserved under the ultraproduct can be found in papers of Arapura \cite{Ar} and Schoutens \cite{Schoutens}.

\section{Proof of the Main Theorem}

Now that we have established some basic knowledge of ultraproducts we can prove the main theorem of this paper, Theorem \ref{thm_1}, in a way similar to Ax and Kochen's proof of Theorem \ref{thm_AK}.  First we need to recall a theorem of Ax-Kochen and Er\u{s}ov \cite{MR0184930,MR0193086}.

\begin{Theo}\label{prop_ultra}(Ax-Kochen and Er\v{s}ov)  Let $K$ and $K'$
be two Henselian valued fields of residual characteristic zero. Assume their residue fields $k$ and $k'$
and their
value groups $\Gamma$ and $\Gamma'$
are elementary equivalent, that is, they have the same set of true sentences in the language of rings, respectively ordered abelian groups. Then $K$ and
$K'$
are elementary equivalent, that is, they satisfy the same set of formulas in the language of
valued fields.
\end{Theo}

\begin{proof}Proof of Theorem \ref{thm_1}.

Fix a numerical polynomial $P$.  Let $\bQ_p^{nr}$ denote the maximally unramified extension of the $p$-adics and let $S$ be the set of all primes.
Suppose by way of contradiction that there is an infinite subset of
primes $e(P) \subset S$ such that for each $p \in e(P)$ there is
a smooth, projective, rationally connected variety $X_{p} \in \mathbb{P}^{n}$ defined over $\bQ_p^{nr}$ having Hilbert
polynomial $P$  and $X(\bQ_p^{nr})=\emptyset$. Now by Lemma
\ref{lem_ultra} there is a non-principal ultrafilter $\Sigma$ containing $e(P)$, and
we can define the non-principal ultraproduct
\[
K = \prod_{p \in S}\bQ_p^{nr}/\Sigma.
\]
Both $K$ and 
$\C((t))$ are  Henselian valued fields of characteristic zero.   The  residue field of $\C((t))$ is simply $\C$, so algebraically closed of characteristic zero.  The  residue field of $K$ is a non-principal ultraproduct of the algebraic closures of the finite fields $\F$ and is  known to be algebraically closed of  characteristic zero (\cite{MR0184930}, Lemma 4).  Thus, the residue fields of $K$ and $\C((t))$ are elementary equivalent.  
  Note that the lemma of Ax and Kochen requires the ultrafilter to be non-principal which is why we need $e(P)$ to be infinite.   $K$ has value group a non-principal
ultraproduct of the integers, otherwise known as an  ultrapower of $\Z$, hence is elementary equivalent to $\Z$, the value group of $\C((t))$. 
Thus by Theorem \ref{prop_ultra}, $K$ and $\C((t))$ are elementary equivalent.

For the numerical polynomial $P$ fixed above, consider all varieties $X \subset \P^n_k$ with Hilbert polynomial $P$.  It is possible to choose a uniform $M\gg0$ such that the Hilbert function $\varphi_X(M)$ is equal to $P(M)$ for all varieties with Hilbert polynomial $P$ (\cite{Bren} 12.47).  Let $Gr$ denote the Grassmannian of codimension-$P(M)$ subspaces of the space of polynomials of degree $M$ in $n+1$ variables:
$$Gr:= Grass\left( {n+M \choose M} - P(M), k[x_0, \ldots , x_n]_M\right).$$  Let $J(X)_M \subset k[x_0, \ldots , x_n]_M$ denote the degree $M$ polynomials vanishing on $X$.
Then, $J(X)_M$ defines a point in the Grassmannian $Gr$ and the set of all projective varieties with Hilbert polynomial $P$ is parameterized by a projective variety known as the Hilbert scheme,  $\mathcal{H}ilb_P \subset Gr$.  Let $g$ denote the smooth morphism that injects $\mathcal{H}ilb_P$ into $Gr$, $g: \mathcal{H}ilb_P \hookrightarrow Gr$.  Notice that since $P$ is a numerical polynomial, everything here is defined over $\mathbb{Q}$.  Over any field $L$ of characteristic 0, every projective $L$-variety $X$ with Hilbert polynomial $P$ is realized as a fiber $g_{b}$ of $g$ above a point $b=J(X)_M \in Gr$.  Thus, $X$ has an $L$-point if and only if $g_{b}$ has an $L$-point.

If the fiber $g_{b}$ is  rationally connected for some $b\in Gr$, then there is an open neighborhood $b\in U\subset Gr$  such that the fiber $g_{u}$ is  rationally connected  if $u\in U$ (\cite{MR1440180} IV.3.11).  Thus, over $\mathbb{Q}$, the set $Z$ of points $b \in Gr$ such that the fiber $g_{b}$ is  rationally connected is an open subset of $Gr$. Since $Gr$ is a variety and $Z$ is an open subset of $Gr$, over any field $L$ of characteristic zero $Z(L)$ is definable in the ring language.  Furthermore, that definition is actually the same one given over $\mathbb{Q}$.  Also, for any field $L$ of characteristic zero and any  rationally connected variety $X$ over $L$ with Hilbert polynomial $P$, there is a an $L$-point
$z\in Z(L)$ such that $X=g_{z}$ and $X(L)\neq \emptyset \iff g_{z}(L) \neq \emptyset$.  

Now let $H_L$ be the set of points  $b \in Gr(L)$  such that the fiber over $b$ is  rationally connected and contains an $L$-point:
$$H_L= \{ b \in Gr(L) : b \in Z(L), g_{b}(L) \neq \emptyset\}.$$  
Then $H_L \subset Z(L)$ is a definable subset of $Gr(L)$ defined by the following formula over $\mathbb{Q}$  with an existential quantifier: $b\in H_L$ if and only if the formula ``$b\in Z$ and $\exists  x\in \mathcal{H}ilb_P$ such that $g_P(x)=b$" is true in the field $L$.  This formula does not depend on the field $L$, but only on the polynomial $P$.
   
The result of Colliot-Th{\'e}l{\`e}ne  \cite{MR2757627} $\C((t))$ tells us that the formula ``$b\in Z$ and $\exists x\in A_P$ such that $g_P(x)=b$" is true over $\C((t))$.  Then by elementary equivalence it is true over the ultraproduct $K$.    
But {\L}o\v{s}'s Theorem tells us that this statement is false over $K$ by the way we constructed our non-principal ultrafilter $\Sigma$.   Thus we arrive at our contradiction, and we have shown that the set of primes $e(P) \subset S$ such that for each $p \in e(P)$ there is
a smooth, projective, rationally connected variety $X_{p} \in \mathbb{P}^{n}$ defined over $\bQ_p^{nr}$ having Hilbert
polynomial $P$  and $X(\bQ_p^{nr})=\emptyset$ is finite.

\end{proof}

\bibliographystyle{alpha}
\bibliography{BradAKBib}

\begin{thebibliography}{KMM92}

\bibitem[AK65]{MR0184930}
James Ax and Simon Kochen.
\newblock Diophantine problems over local fields. {I}.
\newblock {\em Amer. J. Math.}, 87:605--630, 1965.

\bibitem[Ara11]{Ar}
Donu Arapura.
\newblock Frobenius amplitude, ultraproducts, and vanishing on singular spaces.
\newblock {\em Illinois J. Math.}, 55(4):1367--1384 (2013), 2011.

\bibitem[Cam92]{Ca}
F.~Campana.
\newblock Connexit\'e rationnelle des vari\'et\'es de {F}ano.
\newblock {\em Ann. Sci. \'Ecole Norm. Sup. (4)}, 25(5):539--545, 1992.

\bibitem[CMS03]{MR1948099}
Marc Chardin and Guillermo Moreno-Soc{\'{\i}}as.
\newblock Regularity of lex-segment ideals: some closed formulas and
  applications.
\newblock {\em Proc. Amer. Math. Soc.}, 131(4):1093--1102 (electronic), 2003.

\bibitem[CT87]{CT1}
Jean-Louis Colliot-Th{\'e}l{\`e}ne.
\newblock Arithm\'etique des vari\'et\'es rationnelles et probl\`emes
  birationnels.
\newblock In {\em Proceedings of the {I}nternational {C}ongress of
  {M}athematicians, {V}ol. 1, 2 ({B}erkeley, {C}alif., 1986)}, pages 641--653,
  Providence, RI, 1987. Amer. Math. Soc.

\bibitem[CT11]{MR2757627}
Jean-Louis Colliot-Th{\'e}l{\`e}ne.
\newblock Vari\'et\'es presque rationnelles, leurs points rationnels et leurs
  d\'eg\'en\'erescences.
\newblock In {\em Arithmetic geometry}, volume 2009 of {\em Lecture Notes in
  Math.}, pages 1--44. Springer, Berlin, 2011.

\bibitem[CTM04]{CTM}
Jean-Louis Colliot-Th{\'e}l{\`e}ne and David~A. Madore.
\newblock Surfaces de del {P}ezzo sans point rationnel sur un corps de
  dimension cohomologique un.
\newblock {\em J. Inst. Math. Jussieu}, 3(1):1--16, 2004.

\bibitem[Den16]{denef}
Jan Denef.
\newblock Geometric proofs of theorems of {A}x-{K}ochen and {E}r{\v{s}}ov.
\newblock {\em Amer. J. Math.}, 138(1):181--199, 2016.

\bibitem[dJS03]{MR1981034}
A.~J. de~Jong and J.~Starr.
\newblock Every rationally connected variety over the function field of a curve
  has a rational point.
\newblock {\em Amer. J. Math.}, 125(3):567--580, 2003.

\bibitem[Er{\v{s}}65]{MR0193086}
Ju.~L. Er{\v{s}}ov.
\newblock On elementary theory of maximal normalized fields.
\newblock {\em Algebra i Logika Sem.}, 4(3):31--70, 1965.

\bibitem[Esn03]{MR1943746}
H{\'e}l{\`e}ne Esnault.
\newblock Varieties over a finite field with trivial {C}how group of 0-cycles
  have a rational point.
\newblock {\em Invent. Math.}, 151(1):187--191, 2003.

\bibitem[FJ08]{FieldArith}
Michael~D. Fried and Moshe Jarden.
\newblock {\em Field arithmetic}, volume~11 of {\em Ergebnisse der Mathematik
  und ihrer Grenzgebiete. 3. Folge. A Series of Modern Surveys in Mathematics
  [Results in Mathematics and Related Areas. 3rd Series. A Series of Modern
  Surveys in Mathematics]}.
\newblock Springer-Verlag, Berlin, third edition, 2008.
\newblock Revised by Jarden.

\bibitem[GHS03]{MR1937199}
Tom Graber, Joe Harris, and Jason Starr.
\newblock Families of rationally connected varieties.
\newblock {\em J. Amer. Math. Soc.}, 16(1):57--67 (electronic), 2003.

\bibitem[Gre66]{MR0191897}
Marvin~J. Greenberg.
\newblock Rational points in {H}enselian discrete valuation rings.
\newblock {\em Bull. Amer. Math. Soc.}, 72:713--714, 1966.

\bibitem[Har77]{Hart}
Robin Hartshorne.
\newblock {\em Algebraic geometry}.
\newblock Springer-Verlag, New York-Heidelberg, 1977.
\newblock Graduate Texts in Mathematics, No. 52.

\bibitem[Has07]{Bren}
Brendan Hassett.
\newblock {\em Introduction to algebraic geometry}.
\newblock Cambridge University Press, Cambridge, 2007.

\bibitem[KMM92]{KMM}
J{\'a}nos Koll{\'a}r, Yoichi Miyaoka, and Shigefumi Mori.
\newblock Rational connectedness and boundedness of {F}ano manifolds.
\newblock {\em J. Differential Geom.}, 36(3):765--779, 1992.

\bibitem[Koc75]{MR0568318}
Simon Kochen.
\newblock The model theory of local fields.
\newblock In {\em {$1xsy) $} {ISILC} {L}ogic {C}onference ({P}roc. {I}nternat.
  {S}ummer {I}nst. and {L}ogic {C}olloq., {K}iel, 1974)}, pages 384--425.
  Lecture Notes in Math., Vol. 499. Springer, Berlin, 1975.

\bibitem[Kol96]{MR1440180}
J{\'a}nos Koll{\'a}r.
\newblock {\em Rational curves on algebraic varieties}, volume~32 of {\em
  Ergebnisse der Mathematik und ihrer Grenzgebiete. 3. Folge. A Series of
  Modern Surveys in Mathematics [Results in Mathematics and Related Areas. 3rd
  Series. A Series of Modern Surveys in Mathematics]}.
\newblock Springer-Verlag, Berlin, 1996.

\bibitem[Lan52]{MR0046388}
Serge Lang.
\newblock On quasi algebraic closure.
\newblock {\em Ann. of Math. (2)}, 55:373--390, 1952.

\bibitem[Man66]{Manin}
Ju.~I. Manin.
\newblock Rational surfaces over perfect fields.
\newblock {\em Inst. Hautes \'Etudes Sci. Publ. Math.}, (30):55--113, 1966.

\bibitem[Sch05]{Schoutens}
Hans Schoutens.
\newblock Log-terminal singularities and vanishing theorems via non-standard
  tight closure.
\newblock {\em J. Algebraic Geom.}, 14(2):357--390, 2005.

\bibitem[Ser73]{MR0344216}
J.-P. Serre.
\newblock {\em A course in arithmetic}.
\newblock Springer-Verlag, New York, 1973.
\newblock Translated from the French, Graduate Texts in Mathematics, No. 7.

\bibitem[SK79]{MR0526513}
Tetsuji Shioda and Toshiyuki Katsura.
\newblock On {F}ermat varieties.
\newblock {\em T\^ohoku Math. J. (2)}, 31(1):97--115, 1979.

\bibitem[Ter66]{MR0197450}
Guy Terjanian.
\newblock Un contre-exemple \`a une conjecture d'{A}rtin.
\newblock {\em C. R. Acad. Sci. Paris S\'er. A-B}, 262:A612, 1966.

\bibitem[Wit10]{Witt}
Olivier Wittenberg.
\newblock La connexit\'e rationnelle en arithm\'etique.
\newblock In {\em Vari\'et\'es rationnellement connexes: aspects
  g\'eom\'etriques et arithm\'etiques}, volume~31 of {\em Panor. Synth\`eses},
  pages 61--114. Soc. Math. France, Paris, 2010.

\end{thebibliography}

\end{document}